\documentclass[reqno,11pt]{amsart}

\usepackage{mathdots}
\usepackage{tikz}
\usepackage{tikz-cd}
\usetikzlibrary{automata}
\usepackage{amssymb}
\usepackage{amsgen}
\usepackage{amsmath}
\usepackage{amsthm}
\usepackage{mathrsfs}
\usepackage{cite}
\usepackage{amsfonts}
\usepackage{enumitem}
\usepackage{stmaryrd}

\newcommand{\sour}{\mathop{\boldsymbol s}\nolimits}

\newcommand{\ran}{\mathop{\boldsymbol r}\nolimits}

\newcommand{\Ind}{\mathop{\mathrm{Ind}}\nolimits}

\newcommand{\inv}{^{-1}}
\newcommand{\p}{\varphi}

\newcommand{\Tor}{\mathop{\mathrm{Tor}}\nolimits}

\usepackage{xcolor}

\newtheorem{Thm}{Theorem}[section]
\newtheorem{Prop}[Thm]{Proposition}

\newtheorem{Lemma}[Thm]{Lemma}
{\theoremstyle{definition}
}
{\theoremstyle{remark}
}
\newtheorem{Cor}[Thm]{Corollary}
{\theoremstyle{remark}
}
{\theoremstyle{remark}
}

{\theoremstyle{remark}
}
{\theoremstyle{remark}
}
{\theoremstyle{remark}
}
{\theoremstyle{remark}
\newtheorem*{Claim*}{Claim}}
\newtheorem*{theorema}{Theorem A}
\newtheorem*{theoremb}{Theorem B}

\numberwithin{equation}{section}

\title{A homological characterization of AF groupoids}

\author{Benjamin Steinberg}
\address[B.~Steinberg]{%
    Department of Mathematics\\
    City College of New York\\
    Marshak Science Building, Room 529\\
    160 Convent Avenue\\
    New York, New York 10031\\
    USA}
\email{bsteinberg@ccny.cuny.edu}

\thanks{The author was supported by the NSF grant DMS-2452324, a Simons Foundation Collaboration Grant, award number 849561, the Australian Research Council Grant DP230103184, and Marsden Fund Grant MFP-VUW2411.}
\date{\today}

\keywords{}
\subjclass[2020]{20J05, 22A22}
\begin{document}

\begin{abstract}
An ample groupoid is said to be AF if it is a directed union of compact open principal subgroupoids. In this paper, we provide a complete homological characterization of these groupoids. Specifically, we prove that an ample groupoid is AF if and only if it has homological dimension zero.  More generally, we characterize groupoids of homological dimension zero over a unital ring $R$.  
\end{abstract}

\maketitle
 
\section{Introduction}
An ample groupoid is AF (approximately finite) if it is a directed union of compact open principal subgroupoids~\cite{AFequiv,renaultaf,GPS,DeaconuAFdef,Matui}.  Note that the $C^*$-algebra associated with an AF groupoid is itself an AF algebra.     The significance of these groupoids was first recognized in the seminal work of Giordano, Putnam, and Skau~\cite{GPS} on Cantor dynamical systems.

Foundational work of Matui~\cite{Matui} highlighted an important homological feature: the homology of an AF groupoid with integer coefficients vanishes in all dimensions $n > 0$, and the homology in degree $0$ recovers the $K_0$-group of the $C^*$-algebra.
  This property was instrumental in Matui’s computation of the homology for graph groupoids of finite graphs without sources.

In this paper, we prove that AF groupoids are characterized by a stronger homological property,  namely that their homology with coefficients in \emph{any} module vanishes in dimensions greater than $0$. %, that is, we prove the following theorem.

\begin{theorema}
    An ample groupoid $\mathcal G$ is AF if and only if it is of homological dimension $0$.
\end{theorema}

More generally, %if $R$ is a unital ring, then 
the $R$-homological dimension of an ample groupoid $\mathcal G$ with respect to a unital ring $R$ is $0$ if its homology with coefficients in any $R\mathcal G$-module vanishes in dimensions greater than $0$.  Theorem~A is then the special case $R=\mathbb Z$ of the following more general result.

\begin{theoremb}
    Let $R$ be a unital ring and $\mathcal G$ an ample groupoid.  Then $\mathcal G$ has $R$-homological dimension $0$ if and only if the following two conditions hold:
    \begin{enumerate}
        \item $\mathcal G$ is a directed union of quasi-compact open subgroupoids;
        \item The order of any finite subgroup of an isotropy group is a unit in $R$. 
    \end{enumerate}
\end{theoremb}

The paper is organized as follows.  After a preliminary section on ample groupoids and their algebras, we recall the definition of the homology of a groupoid with coefficients in a module and Shapiro's lemma.  We establish continuity of homology with respect to direct limits in this context, as well as discussing Morita invariance of homology with coefficients. A long exact sequence in homology associated to an open invariant subspace of the unit space is proved. This theory is then applied to compact principal ample groupoids, exploiting their Morita equivalence to a space~\cite{GPS}, to compute their homology.   The next section brings everything together to prove Theorem~A based on an idea of the author and van Wyk from their characterization of ample groupoids with a von Neumann regular algebra~\cite{regulargpd}.
The final section proves Theorem~B by analyzing the structure of quasi-compact groupoids. 

\subsection*{Acknowledgments}
An email discussion with Christian B\"onicke first suggested a connection between AF groupoids and the results of~\cite{regulargpd}.

\section{Ample groupoids and their algebras}
Following Bourbaki~\cite{Bourbakitop1}, we call a space $X$ \emph{quasi-compact}  if each open covering of $X$ has a finite subcover and reserve the term \emph{compact} for quasi-compact Hausdorff spaces.  

An \emph{ample groupoid} is a topological groupoid $\mathcal G$ such that the unit space $\mathcal G^0$ is locally compact, Hausdorff and totally disconnected and the source map $\sour\colon \mathcal G\to \mathcal G^0$ is a local homeomorphism.  We do not assume that $\mathcal G$ is Hausdorff. The range map is denoted by $\ran$.

If $X,Y\subseteq \mathcal G^0$, then $\mathcal G_X = \sour\inv(X)$, $\mathcal G^Y = \ran\inv(Y)$ and $\mathcal G^Y_X = \mathcal G^Y\cap \mathcal G_X$.  
An element $g\in \mathcal G$ is \emph{isotropy} if $\sour(g)=\ran(g)$.  The group $\mathcal G^x_x$ is called the \emph{isotropy group} at $x\in \mathcal G^0$.
A \emph{group bundle} is a groupoid consisting solely of %entirely of 
isotropy.

The \emph{orbit} of $x\in \mathcal G^0$ is $\ran(\sour\inv(x))$, i.e., it 
consists of  %all $y\in \mathcal G^0$ 
those $y\in \mathcal G^0$ with an arrow from $x$ to $y$.  
The \emph{orbit space} of $\mathcal G$ is $\mathcal G^0/\mathcal R$ where $\mathcal R$ is the orbit equivalence relation.   
A subset $Y\subseteq \mathcal G^0$ is called \emph{full} if it meets each orbit, that is, $\ran(\mathcal G_Y)=\mathcal G^0$. A subset $X\subseteq \mathcal G^0$ is \emph{invariant} if it is a union of orbits.

 A groupoid is \emph{principal} if it has trivial isotropy.  Note that a principal ample groupoid is Hausdorff.  Indeed, if $g\neq h\in \mathcal G$, then either $\sour(h)\neq \sour(g)$ or $\ran(g)\neq \ran(h)$.  In the former case, we can choose disjoint neighborhoods $U,V$ of $\sour(g)$ and $\sour(h)$ in $\mathcal G^0$, and then $\sour\inv(U)$ and $\sour\inv(V)$ give disjoint neighborhoods of $g,h$.  The same argument works if $\ran(g)\neq \ran(h)$.

An ample groupoid is said to be \emph{AF} if it is a directed union of compact open principal subgroupoids~\cite{DeaconuAFdef}.  Note that an $\mathrm{AF}$ groupoid is principal and hence Hausdorff.  AF groupoids were extensively studied in~\cite{renaultaf,GPS,Matui}.  

%\begin{Prop}\label{p:directed.af}
%    A  directed union of open AF subgroupoids is AF.
%\end{Prop}
%\begin{proof}
%    Suppose $\mathcal G=\bigcup_{i\in D}\mathcal G_i$ is a directed union of AF subgroupoids and $\mathcal G_i=\bigcup_{j\in D_i}\mathcal G_{i,j}$ is a directed union of compact open principal subgroupoids of $\mathcal G_i$.  Then $\mathcal G=\bigcup_{i\in D}\bigcup_{j\in D_i}\mathcal G_{i,j}$ is a union of compact open principal subgroupoids.  The union is directed because if $\mathcal G_{i,j}$ and $\mathcal G_{k,\ell}$ are given, we can first find $\mathcal G_a$ containing $\mathcal G_i\cup \mathcal G_k$.  Then by compactness of $\mathcal G_{i,j}\cup \mathcal G_{k,\ell}$, we can find $b\in D_a$ with $\mathcal G_{i,j}\cup \mathcal G_{k,\ell}\subseteq \mathcal G_{a,b}$.  Therefore, the union is directed. 
%\end{proof}

The definition of an AF groupoid in~\cite{GPS} looks different than the one we use.  The authors of~\cite{GPS} define an ample groupoid to be AF if it is principal, and it is a directed union of open subgroupoids $\mathcal H_i$ such that $\mathcal G^0\subseteq \mathcal H_i$ and $\mathcal H_i\setminus \mathcal G^0$ is compact.  These two definitions are readily seen to be equivalent. If $\mathcal G=\bigcup \mathcal H_i$ is a directed union of compact open principal groupoids, then $\mathcal G$ is principal, and putting $\mathcal H_i'=\mathcal G^0\cup \mathcal H_i$, we have $\mathcal G=\bigcup \mathcal H_i'$ where each $\mathcal H_i'$ is open, contains $\mathcal G^0$ and $\mathcal H_i'\setminus \mathcal G^0=\mathcal H_i\setminus \mathcal H_i^0$ is compact.  Conversely, it is shown in~\cite[Lemma~3.4]{GPS} that if $\mathcal H$ is principal and $\mathcal H\setminus \mathcal H^0$ is compact, then $\mathcal H$ is a disjoint union $Y\sqcup \mathcal H'$ of two clopen subgroupoids  where $Y$ is a locally compact totally disconnected space and $\mathcal H'$ is a compact principal groupoid.  Clearly, $\mathcal H$ is then AF in the sense we are using, as $Y$ is a directed union of compact open subsets.   Since a directed union of open AF subgroupoids in our sense is clearly AF, an
 AF groupoid in the sense of~\cite{GPS} is AF in ours.

A \emph{bisection} of $\mathcal G$ is an open\footnote{Some authors do not require bisections to be open.} subset $U$ such that $\sour|_U,\ran|_U$ are injective.  The compact bisections form a basis for the topology on an ample groupoid $\mathcal G$ and they also form an inverse semigroup under setwise product with pointwise inverse for the involution. The reader is referred to~\cite{Lawson} for  inverse semigroups.

Following~\cite{regulargpd} we call an ample groupoid $\mathcal G$ \emph{uniformly bounded} if there is an integer $M\geq 1$ with $|\ran^{-1}(x)|\leq M$ for all $x\in \mathcal G^0$; note that $M$ bounds the size of each isotropy group and orbit of $\mathcal G$.  Every quasi-compact groupoid is uniformly bounded~\cite[Lemma~3.3]{regulargpd}. If $\mathcal G$ is a quasi-compact groupoid, the unit space $\mathcal G^0=\sour(\mathcal G)$ is compact. 

%or, equivalently, $|\sour^{-1}(x)|\leq M$ for all $x\in \mathcal G^0$.  
%One says that $\mathcal G$ is \emph{approximately uniformly bounded} if it is a directed union of uniformly bounded open subgroupoids. % and that it is \emph{approximately quasi-compact} if it is a directed union of quasi-compact open subgroupoids. 

The following proposition is part of~\cite[Proposition~3.4]{regulargpd}.

\begin{Prop}\label{p:uniformly.bdd.vs.local.finite}
 The following are equivalent for an ample groupoid $\mathcal G$.
\begin{enumerate}
  \item $\mathcal G$ is a directed union of quasi-compact open subgroupoids.
  \item $\mathcal G$ is a directed union of uniformly bounded open subgroupoids.
  \item The inverse semigroup of compact bisections of $\mathcal G$ is locally finite.
\end{enumerate}
\end{Prop}

If $X$ is a topological space with a basis of compact open sets and $R$ is a unital ring, then $RX$ is the left $R$-submodule of $R^X$ generated by the characteristic functions $1_U$ of compact open subsets $U$ of $X$; in fact, $RX$ is naturally an $R$-bimodule.  When $X$ is Hausdorff, this is the collection of locally constant mappings $f\colon X\to R$ with compact support.

If $\p\colon X\to Y$ is a local homeomorphism, then $\p_*\colon RX\to RY$ defined by 
\[\p_*(f)(y)=\sum_{\p(x)=y}f(x)\]
is a homomorphism. 

If $\mathcal G$ is an ample groupoid,  $R\mathcal G$ is a ring~\cite{mygroupoidalgebra} with respect to convolution 
\[f\ast h(g) = \sum_{\ran(g)=\ran(a)}f(a)h(a\inv g).\] 
We often just write $fh$ for the product. 
%There is an involution on $\mathbb Z\mathcal G$, fixing $\mathbb Z\mathcal G^0$, given by $f^*(g) = f(g\inv)$.
%Therefore, any right $\mathbb Z\mathcal G$-module $M$ can be made into a left module via $fm=mf^*$.
As a left $R$-module $R\mathcal G$ is spanned by the characteristic functions of compact bisections~\cite{mygroupoidalgebra}.   

It was proved by Li~\cite[Corollary~2.3]{XLi} that $R\otimes_{\mathbb Z}\mathbb ZX\cong RX$  via the multiplication map for any space $X$, and hence $R\otimes_{\mathbb Z}\mathbb Z\mathcal G\cong R\mathcal G$ as rings.  In particular, an $R\mathcal G$-module is the same thing as a $\mathbb Z\mathcal G$-module with an $R$-module structure that commutes with the $\mathbb Z\mathcal G$-module structure.

We shall need later the following technical result.

\begin{Prop}\label{p:generate}
    Let $S$ be an inverse semigroup of compact bisections of $\mathcal G$ such that $\mathcal G=\bigcup S$.   Then  $R\mathcal G$ is spanned by the characteristic functions $1_V$ with $V$ compact open and $V\subseteq U$ for some $U\in S$.  Moreover, if $S$ is finitely generated, then $\mathcal G^0$ is compact. 
\end{Prop}
\begin{proof}
Using that $R\mathcal G=R\otimes_{\mathbb Z}\mathbb Z\mathcal G$, it suffices to prove the first statement for $R=\mathbb Z$.
Let $T$ be the inverse semigroup of compact bisections contained in an element of $S$.   It is  a basis for the topology of $\mathcal G$.  Indeed,  if $U$ is an open subset of $\mathcal G$ and $g\in U$, then there is $V\in S$ with $g\in V$.  Then $V\cap U$ is an open neighborhood of $g$, and so we can find $g\in K\subseteq U\cap V$ with $K$ compact open.  But then  $K\in T$.  In particular, since $\mathcal G^0$ is open, it has a basis of elements of  $T$.  Such elements are idempotents of $T$, and so the idempotents of $T$ generate the boolean algebra of compact open subsets of $\mathcal G^0$.  The first statement now follows from~\cite[Proposition~4.14]{mygroupoidalgebra}.

For the final statement, let $X$ be a finite generating set for $S$ and put $U = \bigcup_{V\in X\cup X\inv}VV\inv$.  Note that $U\subseteq \mathcal G^0$  is compact open. Then if $V\in S$, we have that $V=V_1\cdots V_n$ with $V_1,\ldots, V_n\in X\cup X\inv$.  Then $UV_1=V_1$, and so $UV=V$.  Now if $x\in \mathcal G^0$, then $x\in V$ for some $V\in S$, and so $x\in UV$.  It follows that $x\in U$, and so $\mathcal G^0\subseteq U$.  Thus $\mathcal G^0=U$ is compact.
\end{proof}

If $X$ is a space with a basis of compact open sets, $R$ is a unital ring, $U\subseteq X$ is open and $C=X\setminus U$, then there is an exact sequence
\[0\to RU\to RX\to RC\to 0\] where the first map is extension by $0$ and the second map is restriction; see~\cite[Proposition~4.3]{millersteinberg}.   Moreover, if $\mathcal G$ is an ample groupoid, $U\subseteq \mathcal G^0$ is an open invariant subset and $C=\mathcal G^0\setminus U$, then $R\mathcal G_U^U$ is an ideal of $R\mathcal G$ and $R\mathcal G^C_C\cong R\mathcal G/R\mathcal G_U^U$ via the restriction map $R\mathcal G\to R\mathcal G^C_C$; see~\cite[Corollary~4.4]{millersteinberg}.

%In fact, if $S$ is the inverse semigroup of compact bisections of $\mathcal G$, then $R\mathcal G\cong RS/I$ via $1_U\mapsfrom U$, where $I$ is the ideal generated by all elements of the form $U+V-U\cup V$ with $U,V\subseteq \mathcal G^0$ compact open and disjoint~\cite[Corollary~2.15]{simplicity} (which is fully justified in~\cite[Corollary~4.4]{millersteinberg}).

%\section{Groupoid algebras}

\section{Homology of groupoids}
Let $\mathcal G$ be an ample groupoid and $R$ a unital ring. A left $R\mathcal G$-module $M$ is \emph{unitary} if $R\mathcal G\cdot M=M$, and dually for right modules.  The unitary  $R\mathcal G$-modules form an abelian category with enough projectives. 

The trivial right $R\mathcal G$-module is $R\mathcal G^0$ with module action 
\[hf(x) = \sum_{\sour(g)=x}h(\ran(g))f(g)\] for $f\in R\mathcal G$ and $h\in  R\mathcal G^0$.  The trivial left $R\mathcal G$-module is $R\mathcal G^0$ with the left action defined dually. These modules are both unitary.

If $M$ is a unitary left $\mathbb Z\mathcal G$-module, then the homology of $\mathcal G$ with coefficients in $M$ is defined as 
\[H_n(\mathcal G,M) = \Tor_n^{\mathbb Z\mathcal G}(\mathbb Z\mathcal G^0,M).\]
One sets $H_n(\mathcal G) = H_n(\mathcal G,\mathbb Z\mathcal G^0)=\Tor_n^{\mathbb Z\mathcal G}(\mathbb Z\mathcal G^0,\mathbb Z\mathcal G^0)$.

The \emph{homological dimension} of $\mathcal G$ is $\sup\{n\mid H_n(\mathcal G,-)\neq 0\}$.   Equivalently, it is the shortest length of a flat resolution (by unitary modules) of the right $\mathbb Z\mathcal G$-module $\mathbb Z\mathcal G^0$~\cite[Proposition~8.17]{Rotmanhom}.  In particular, the homological dimension of $\mathcal G$ is $0$ if and only if $\mathbb Z\mathcal G^0$ is flat.  More generally, the \emph{$R$-homological dimension} of $\mathcal G$ is the supremum of the integers $n$ such that $H_n(\mathcal G,M)\neq 0$ for a unitary $R\mathcal G$-module $M$.  We shall see later that this is the shortest length of a flat resolution of the right $R\mathcal G$-module $R\mathcal G^0$.

There is a standard flat resolution $B_\bullet(\mathcal G)$ of the right $\mathbb Z\mathcal G$-module $\mathbb Z\mathcal G^0$, known as the bar resolution.   Consider the space \[\mathcal G^{n}=\mathcal G\times_{\mathcal G^0}\overbrace{
\cdots}^{\times n} \times_{\mathcal G^0}\mathcal G.\]  
%The right module $\mathbb Z\mathcal G^{n+1}\cong \mathbb Z\mathcal G^n\otimes_{\mathbb Z\mathcal G^0}\mathbb Z\mathcal G$ (see~\cite[Proposition~2.9]{AMiller}) is flat for $n\geq 0$~\cite[Proposition~2.8]{AMiller}, 
The right module $\mathbb Z\mathcal G^{n+1}$ is flat for $n\geq 0$~\cite[Proposition~2.4]{XLi} where the module structure is given for $f\in \mathbb Z\mathcal G$ and $h\in \mathbb Z\mathcal G^{n+1}$ by 
\[hf(g_0,\ldots, g_n)=\sum_{\ran(a)=\ran(g_n)}h(g_0,\ldots, g_{n-1},a)f(a\inv g_n).\]
Moreover, we have a flat resolution
\[B_\bullet(\mathcal G):\quad\cdots\to \mathbb Z\mathcal G^2\to \mathbb Z\mathcal G^1\to \mathbb Z\mathcal G^0\to 0\]
where the boundary map $d_n\colon \mathbb Z\mathcal G^{n+1}\to \mathbb Z\mathcal G^n$ is defined as follows. First of all $d_0=\sour_*$. For $n\geq 1$, \[d_n = \sum_{i=0}^n(-1)^i(d_i^n)_*\] where $d_i^n\colon \mathcal G^{n+1}\to \mathcal G^n$ is the local homeomorphism given by 
\[d_i^n(g_0,\ldots,g_n)=\begin{cases}(g_0,\ldots, g_{i-1}g_i,\ldots, g_n), & \text{if}\ 1\leq i\leq n\\ (g_1,\ldots, g_n), & \text{if}\ i=0.\end{cases}\] Note that the $d_i^n$ are $\mathcal G$-equivariant maps of right $\mathcal G$-spaces, as is $\sour\colon \mathcal G^1\to \mathcal G^0$.  Therefore, they induce right module homomorphisms. %The maps $(-1)^n(h_n)_*$ induce a contracting homotopy at the level of abelian groups where $h_0\colon \mathcal G^0\to \mathcal G$ is the inclusion and $h_n\colon \mathcal G^n\to \mathcal G^{n+1}$ is the local homeomorphism given by $(g_0,\ldots,g_n)\mapsto (g_0,\ldots,g_n,\sour(g_n))$ for $n\geq 1$. 
See~\cite[Example~2.14]{AMiller} for more details, where left modules are considered.

If $\mathcal H$ is an open subgroupoid of $\mathcal G$, then any unitary left $\mathbb Z\mathcal G$-module $M$ can be restricted to a unitary $\mathbb Z\mathcal H$-module $\mathbb Z\mathcal H\cdot M$.  There is a left adjoint to restriction called induction.  If $M$ is a unitary left $\mathbb Z\mathcal H$-module, then $\Ind_\mathcal H^\mathcal G M = \mathbb Z\mathcal G_{\mathcal H^0}\otimes_{\mathbb Z\mathcal H}M$.  Miller~\cite[Lemma~2.19]{AMiller} proves an analogue of Shapiro's lemma in this context.

\begin{Lemma}[Shapiro's Lemma]   
Let $\mathcal G$ be an ample groupoid and $\mathcal H$ an open subgroupoid.  Then, for any unitary $\mathbb Z\mathcal H$-module $M$, there is an isomorphism
\[H_n(\mathcal G,\Ind_\mathcal H^\mathcal G M)\cong H_n(\mathcal H,M)\]
of abelian groups.
\end{Lemma}

Note that if $M$ is an $R\mathcal H$-module, then $\Ind_{\mathcal H}^{\mathcal G} M$ is an $R\mathcal G$-module via the $R$-module structure on $M$. It immediately follows that the $R$-homological dimension of an open subgroupoid $\mathcal H\subseteq \mathcal G$ cannot exceed that of $\mathcal G$.

\begin{Cor}\label{c:shapiro.cor}
    Let $\mathcal G$ be an ample groupoid and $\mathcal H$ an open subgroupoid.  Then the $R$-homological dimension of $\mathcal H$ is bounded by that of $\mathcal G$.
\end{Cor}

The following theorem establishes the continuity of groupoid homology.  Note that if $R=\varinjlim_D R_i$ is a direct limit of rings, and $M=\varinjlim_D M_i$, $N=\varinjlim_D N_i$ with the $M_i$ right $R_i$-modules and the $N_i$ left $R_i$-modules, then $M\otimes_R N\cong \varinjlim _D (M_i\otimes_{R_i} N_i)$.  Indeed, for a ring $S$ and right/left $S$-modules $A,B$, we have that
\[A\otimes_S B=\mathrm{coeq}(A\otimes_{\mathbb Z} S\otimes_{\mathbb Z} B\rightrightarrows  A\otimes_{\mathbb Z} B)\] is the coequalizer\footnote{The coequalizer of two homomorphisms is the cokernel of their difference.} of the maps $a\otimes s\otimes b\mapsto as\otimes b$ and $a\otimes s\otimes b\mapsto a\otimes sb$ induced by the right and left module actions of $S$, and coequalizers commute with direct limits. (Note that the additive group of a direct limit of rings is the direct limit of the additive groups.)  

\begin{Thm}\label{t:dlim.hom}
Suppose that $\mathcal G=\bigcup_{i\in D}\mathcal H_i$ is a directed union of open subgroupoids $\mathcal H_i$.   Then 
\[H_n(\mathcal G,M)\cong \varinjlim_{i\in D}\ H_n(\mathcal H_i,\mathbb Z\mathcal H_i\cdot M)\] for any unitary left $\mathbb Z\mathcal G$-module $M$. In particular, $H_n(\mathcal G) = \varinjlim_D H_n(\mathcal H_i)$.
\end{Thm}
\begin{proof}
    Notice that $\mathcal G^n=\bigcup_{i\in D}\mathcal H^n_i$ is a directed union of open subspaces for $n\geq 0$.  It follows that $\mathbb Z\mathcal G^n=\bigcup_{i\in D} \mathbb Z\mathcal H_i^n$ is a directed union with each $\mathbb Z\mathcal H^n_i$ a unitary right $\mathbb Z\mathcal H_i$-module.  Also, we have that $M = \mathbb Z\mathcal G\cdot M=\bigcup_{i\in D}\mathbb Z\mathcal H_i\cdot M$. Noting that the bar resolution is functorial, %Thus $B_\bullet(\mathcal G)\cong \bigcup_{i\in D}B_\bullet(\mathcal H_i)$.   
    the discussion preceding the theorem implies that \[B_\bullet(\mathcal G)\otimes_{\mathbb Z\mathcal G} M\cong \varinjlim_{i\in D} (B_\bullet(\mathcal H_i)\otimes_{\mathbb Z\mathcal H_i}\mathbb Z\mathcal H_i\cdot M).\]  Since $\varinjlim_{i\in D}$ is an exact functor, as $D$ is directed, taking homology commutes with $\varinjlim_{i\in D}$.  Therefore, we have that  $H_n(\mathcal G,M)=H_n(B_\bullet(\mathcal G)\otimes_{\mathbb Z\mathcal G} M)\cong \varinjlim_{i\in D} H_n(B_\bullet(\mathcal H_i)\otimes_{\mathbb Z\mathcal H_i} \mathbb Z\mathcal H_i\cdot M)=\varinjlim_{i\in D}H_n(\mathcal H_i, \mathbb Z\mathcal H_i\cdot M)$, as required.  

    The final statement follows on observing that $\mathbb Z\mathcal H_i\cdot \mathbb Z\mathcal G^0 = \mathbb Z\mathcal H_i^0$.
\end{proof}

Considering in Theorem~\ref{t:dlim.hom} the special case where each $H_n(\mathcal H_i,\mathbb Z\mathcal H_i\cdot M)=0$ and noting that $\mathbb Z\mathcal H_i\cdot M$ is an $R\mathcal H_i$-module when $M$ is an $R\mathcal G$-module, we obtain the following corollary.

\begin{Cor}\label{c:dir.hom.zero}
    Let $\mathcal G$ be a directed union of open subgroupoids of $R$-ho\-mo\-lo\-gi\-cal dimension $0$.  Then $\mathcal G$ has $R$-homological dimension $0$.
\end{Cor}

If $R$ is a unital ring, then an idempotent $e\in R$ is \emph{full} if $ReR=R$.  In this case, $R$ and $eRe$ are Morita equivalent, and the natural maps $Re\otimes_{eRe}eR\to ReR=R$ and $eR\otimes_R Re\to eRe$ are bimodule isomorphisms~\cite[Example (18.30)]{Lam2}.  It follows that the functor $M\mapsto Me$ from right $R$-modules to right $eRe$-modules is an equivalence and hence preserves flatness.  Moreover, if $A$ is a right $R$-module and $B$ is a left $A$-module, then \[Ae\otimes_{eRe}eB\cong A\otimes_{R} Re\otimes_{eRe}eR\otimes_RB \cong A\otimes_R R\otimes_R B\cong A\otimes_R B.\]  It follows that $\Tor^{eRe}_n(Ae,eB)\cong \Tor^R_n(A,B)$ for $A$ and $B$ as above.  Indeed, if $F_\bullet\to A$ is a flat resolution, then $F_\bullet e\to Ae$ is a flat resolution and $\Tor^R_n(A,B)\cong H_n(F_\bullet\otimes_R B)\cong H_n(F_\bullet e\otimes_{eRe}eB)\cong \Tor^{eRe}_n(Ae,eB)$.

The following theorem can be deduced from~\cite[Corollary~4.6]{CMhomology}, but we give a direct and more elementary proof.

\begin{Thm}\label{t:cheap.Morita}
    Let $\mathcal G$ be an ample groupoid with compact unit space and let $Y\subseteq\mathcal G^0$ be compact open and full.  Then for any left $\mathbb ZG$-module $M$, $H_n(\mathcal G,M)\cong H_n(\mathcal G^Y_Y,1_YM)$. In particular, $H_n(\mathcal G)\cong H_n(\mathcal G^Y_Y)$.
\end{Thm}
\begin{proof}
 First, observe that $1_Y\mathbb Z\mathcal G1_Y=\mathbb Z\mathcal G_Y^Y$, viewed as a subring via extension by $0$.  We now show that $1_Y$ is a full idempotent.   
    
    Indeed, since $\mathcal G_Y=\sour\inv(Y)$ is open, we can cover it by compact bisections.  Since $\mathcal G^0$ is compact and $\ran(\mathcal G_Y)=\mathcal G^0$ by fullness, we can find compact bisections $V_1,\ldots, V_n$ contained in $\mathcal G_Y$ such that $\mathcal G^0=\ran(\mathcal G_Y) =\bigcup_{i=1}^n\ran(V_i) =  V_1YV_1\inv \cup\cdots\cup V_nYV_n\inv$.   The inclusion-exclusion principle then yields
    \[1_{\mathcal G^0} = \sum_{k=1}^n(-1)^{k+1}\sum_{1\leq i_1<\cdots<i_k\leq n}1_{V_{i_1}}1_Y1_{V_{i_1}\inv}\cdots 1_{V_{i_k}}1_Y1_{V_{i_k}\inv} \in \mathbb Z\mathcal G\cdot 1_Y\cdot \mathbb Z\mathcal G.\] Therefore, $1_Y$ is full.

    Now $\mathbb Z\mathcal G^01_Y=\mathbb ZY$ as a right $\mathbb Z\mathcal G^Y_Y$-module.  Therefore, by the discussion before the  theorem
    \[H_n(\mathcal G,M) = \Tor^{\mathbb Z\mathcal G}_n(\mathbb Z\mathcal G^0, M)\cong \Tor_n^{\mathbb Z\mathcal G^Y_Y}(\mathbb ZY,1_YM)=H_n(\mathcal G^Y_Y,1_YM).\] In particular, taking $M=\mathbb Z\mathcal G^0$  and observing that $1_Y\mathbb Z\mathcal G^0=\mathbb ZY$ yields $H_n(\mathcal G)\cong H_n(\mathcal G_Y^Y)$.
\end{proof}

We now apply this to compact principal ample groupoids following~\cite{GPS}.

\begin{Thm}\label{t:hom.zero.prinipal}
    Let $\mathcal G$ be a compact principal ample groupoid.  Then $\mathcal G$ has homological dimension $0$.
\end{Thm}
\begin{proof}
Let $\mathcal R$ be the orbit equivalence relation on $\mathcal G^0$.  Then we have that $\mathcal R\subseteq \mathcal G^0\times \mathcal G^0$ is the image of $(\ran,\sour)\colon \mathcal G\to \mathcal G^0\times \mathcal G^0$, and hence $\mathcal R$ is closed as $\mathcal G$ is compact and $\mathcal G^0$ is Hausdorff.  Thus $\mathcal G^0/\mathcal R$ is Hausdorff by~\cite[Proposition 10.4.8]{bourbakitop}, hence compact.    We claim that the projection $q\colon \mathcal G^0\to \mathcal G^0/\mathcal R$ is a local homeomorphism (cf.~\cite[p.~451, Comment]{GPS}).  %By~\cite[Lemma~2.10]{AKM22}, it suffices to show that the left action of $\mathcal G$ on $\mathcal G^0$ is basic.  
Recall that a left $\mathcal G$-space $X$ is basic~\cite{AKM22,AMiller} if $\mathcal G\times_{\mathcal G^0} X\to X\times_{\mathcal G\backslash X} X$ given by $(g,x)\mapsto (gx,x)$ is a homeomorphism, in which case $X\to \mathcal{G}\backslash X$ is a local homeomorphism by~\cite[Lemma~2.10]{AKM22}.  The left action of $\mathcal G$ on $\mathcal G^0$ is basic in our case with $\mathcal G\backslash \mathcal G^0=\mathcal G^0/\mathcal R$.  Indeed, identifying $\mathcal G\times_{\mathcal G^0} \mathcal G^0$ with $\mathcal G$ via the projection, this amounts to showing $(\ran,\sour)\colon \mathcal G\to \mathcal R\subseteq \mathcal G^0\times \mathcal G^0$ is a homeomorphism.  The map is surjective by definition.  The map is injective because $\mathcal G$ is principal. But a bijective continuous map between compact  spaces is a homeomorphism.

Since $q\colon \mathcal G^0\to \mathcal G^0/\mathcal R$ is a local homeomorphism and $\mathcal G^0/\mathcal R$ is compact, we deduce that $\mathcal G^0/\mathcal R$ is compact and totally disconnected.  Therefore, there is a section $t\colon \mathcal G^0/\mathcal R\to \mathcal G^0$ of $q$ (we can cover $\mathcal G^0/\mathcal R$ by finitely many disjoint clopen sets on which $q$ has a section and glue together the sections).  Then $Y=t(\mathcal G^0/\mathcal R)$ is a compact open subspace of $\mathcal G^0$ that hits every orbit exactly once and hence is full.  Observe that since  $Y$ intersects each orbit exactly once, if $g\in \mathcal G$ goes between elements of $Y$, then $g$ must be isotropy and hence a unit, as $\mathcal G$ is principal.   It follows that $\mathcal G^Y_Y=Y$.   Since $\mathbb ZY$ is a flat (in fact, projective) $\mathbb ZY$-module, $Y$ has homological dimension zero.   The result then follows from Theorem~\ref{t:cheap.Morita}.  
\end{proof}

The following lemma and theorem will not be used in the proof of Theorem~A, and so can be safely omitted by those only interested in Theorem~A.

\begin{Lemma}\label{l:ideal.nice}
    Let $R$ be a unital ring, $\mathcal G$ be an ample groupoid, $U\subseteq \mathcal G^0$ be open and invariant and $C=\mathcal G^0\setminus U$.  Let $M$ be a left $R\mathcal G$-module.  Then:
    \begin{enumerate}
        \item $R\mathcal G_U^U\otimes_{R\mathcal G} M\cong R\mathcal G^U_U\cdot M$ via the action map.
        \item If $M$ is unitary, $R\mathcal G_C^C\otimes_{R\mathcal G}M\cong M/R\mathcal G^U_U\cdot M$.
    \end{enumerate}
\end{Lemma}
\begin{proof}
    Clearly, the action map $R\mathcal G_U^U\otimes_{R\mathcal G} M\to R\mathcal G^U_U\cdot M$ is a surjective $R\mathcal G^U_U$-module homomorphism.  Suppose that $\sum f_i\otimes m_i$ belongs to the kernel, that is, $\sum f_im_i=0$.  Choose $U\subseteq \mathcal G^0$ compact open such that $1_Uf_i=f_i$ for all $i$.  Then  $\sum f_i\otimes m_i =  \sum 1_Uf_i\otimes m_i= 1_U\otimes \sum f_im_i=0$, as required. 

    For the second item, we identify $R\mathcal G^C_C$ with $R\mathcal G/R\mathcal G^U_U$.  We have a commutative diagram with exact rows and surjective columns
   \[\begin{tikzcd}
  % & & 0\ar{d} & \\
    &    R\mathcal G^U_U\otimes_{R\mathcal G} M\ar{r}\ar{d}& R\mathcal G\otimes_{R\mathcal G}M \ar{r}\ar{d} & R\mathcal G/R\mathcal G_U^U\otimes_{R\mathcal G} M\ar{r}\ar{d} & 0 \\
        0\ar{r} & R\mathcal G^U_U\cdot M\ar{r} & M \ar{r}& M/R\mathcal G^U_U\cdot M%\\
        %& 0 & 0& 0
    \end{tikzcd}\]
    where the middle vertical map is an isomorphism by (1) applied to the open invariant set $\mathcal G^0$ and the third vertical map is $(r+R\mathcal G^U_U)\otimes m\mapsto rm+R\mathcal G^U_U\cdot M$.  The snake lemma now implies that this third map is an isomorphism.  
\end{proof}

Of course, the left-right dual of Lemma~\ref{l:ideal.nice} holds. Notice that taking $U=\mathcal G^0$, we have that $R\mathcal G\otimes_{R\mathcal G}M\cong R\mathcal G\cdot M$.

The following generalizes~\cite[Proposition~4.2]{millersteinberg} to module coefficients.

\begin{Thm}\label{t:exact.seq}
    Let $\mathcal G$ be an ample groupoid and $M$ a unitary left $\mathbb Z\mathcal G$-module.  If $U\subseteq \mathcal G^0$ is an open invariant subset and $C=\mathcal G^0\setminus U$, then there is a long exact sequence
    \begin{align*}
    \cdots&\to H_{n+1}(\mathcal G^C_C,M/\mathbb Z\mathcal G^U_U\cdot M)\to H_n(\mathcal G^U_U,\mathbb ZG^U_U\cdot M)\to H_n(\mathcal G,M)\\ &\to H_n(\mathcal G^C_C,M/\mathbb Z\mathcal G^U_U\cdot M) \to\cdots
    \end{align*}
    in homology.
\end{Thm}
\begin{proof}

    The exact sequence $0\to \mathbb Z\mathcal G^U_U\cdot M\to M\to M/\mathbb Z\mathcal G^U_U\cdot M\to 0$ of left $\mathbb Z\mathcal G$-modules gives rise to a long exact sequence in homology 
    \begin{align*}
    \cdots&\to H_{n+1}(\mathcal G,M/\mathbb Z\mathcal G^U_U\cdot M)\to H_n(\mathcal G,\mathbb ZG^U_U\cdot M)\to H_n(\mathcal G,M)\\ &\to H_n(\mathcal G,M/\mathbb Z\mathcal G^U_U\cdot M) \to\cdots
    \end{align*}
    and so it remains to prove that  $H_n(\mathcal G,\mathbb ZG^U_U\cdot M)\cong  H_n(\mathcal G^U_U,\mathbb ZG^U_U\cdot M)$ and 
    $H_n(\mathcal G,M/\mathbb Z\mathcal G^U_U\cdot M)\cong H_n(\mathcal G^C_C,M/\mathbb Z\mathcal G^U_U\cdot M)$.  

    Observe that $(\mathcal G_U^U)^n$ is an open subset of $\mathcal G^n$ with complement $(\mathcal G^C_C)^n$ using the invariance of $U$. Therefore, $0\to \mathbb Z(\mathcal G^U_U)^n\to \mathbb Z\mathcal G^n\to \mathbb Z(\mathcal G^C_C)^n\to 0$ is exact where the last map is restriction.  Note that $\mathbb Z\mathcal G^n\cdot \mathbb Z\mathcal G^U_U = \mathbb Z(\mathcal G^U_U)^n$ by invariance of $U$, and so $\mathbb Z\mathcal G^n/\mathbb Z(\mathcal G^U_U)^n=\mathbb Z\mathcal G^n/\mathbb Z\mathcal G^n\cdot \mathbb Z\mathcal G^U_U$ is a $\mathbb Z\mathcal G/\mathbb Z\mathcal G^U_U$-module, which can be identified with $\mathbb Z(\mathcal G^C_C)^n$ as a $\mathbb Z\mathcal G^C_C$-module via the restriction map.  Moreover, the restriction map commutes with the maps $(d^n_i)_\ast$ and $\sour_*$ used to define the boundary maps in the bar resolution because every preimage of an element of $\mathcal G^C_C$ under $d_i^n$, or $\sour$, belongs to $\mathcal G^C_C$ by invariance of $C$.   Therefore, $B_\bullet(\mathcal G)\cdot \mathbb Z\mathcal G^U_U=B_\bullet(\mathcal G^U_U)$ as chain complexes of $\mathbb Z\mathcal G^U_U$-modules and $B_\bullet(\mathcal G)/B_\bullet(\mathcal G)\cdot \mathbb Z\mathcal G_U^U\cong B_\bullet(\mathcal G^C_C)$ as chain complexes of $\mathbb Z\mathcal G^C_C$-modules.

Putting this all together, and using several times Lemma~\ref{l:ideal.nice} and its left-right dual, we compute that 
\begin{align*}
B_\bullet(\mathcal G)\otimes_{\mathbb Z\mathcal G}  \mathbb Z\mathcal G^U_U\cdot M &\cong B_\bullet(\mathcal G)\otimes_{\mathbb Z\mathcal G} \mathbb Z\mathcal G^U_U\otimes_{\mathbb Z\mathcal G^U_U} \mathbb Z\mathcal G^U_U\cdot M \\ &\cong B_\bullet(\mathcal G)\cdot \mathbb Z\mathcal G^U_U \otimes_{\mathbb Z\mathcal G^U_U} \mathbb Z\mathcal G^U_U\cdot M\\ &\cong B_\bullet(\mathcal G_U^U)\otimes_{\mathbb Z\mathcal G^U_U} \mathbb Z\mathcal G^U_U\cdot M\\
B_{\bullet}(\mathcal G)\otimes_{\mathbb Z\mathcal G}M/\mathbb Z\mathcal G_U^U\cdot M &\cong B_{\bullet}(\mathcal G)\otimes_{\mathbb Z\mathcal G} \mathbb Z\mathcal G^C_C\otimes_{\mathbb Z\mathcal G^C_C} M/\mathbb Z\mathcal G^U_U\cdot M\\ &\cong B_\bullet(\mathcal G)/B_\bullet(\mathcal G)\cdot \mathbb Z\mathcal G_U^U \otimes_{\mathbb Z\mathcal G^C_C}M/\mathbb Z\mathcal G^U_U\cdot M\\ &\cong B_\bullet(\mathcal G^C_C)\otimes_{\mathbb Z\mathcal G^C_C}M/\mathbb Z\mathcal G^U_U\cdot M.
\end{align*}
Computing homology with the bar resolution yields $H_n(\mathcal G,\mathbb ZG^U_U\cdot M)\cong  H_n(\mathcal G^U_U,\mathbb ZG^U_U\cdot M)$ and 
    $H_n(\mathcal G,M/\mathbb Z\mathcal G^U_U\cdot M)\cong H_n(\mathcal G^C_C,M/\mathbb Z\mathcal G^U_U\cdot M)$.  
\end{proof}

\section{Proof of Theorem~A}
We need two lemmas to prove Theorem~A; they will also be used in the proof of Theorem~B.

\begin{Lemma}\label{l:generating.set}
Let $\mathcal G$ be an ample groupoid such that $\mathcal G=\bigcup S$ with $S$ an inverse semigroup of compact bisections.  Suppose that $X\subseteq S$ generates $S$ as  an inverse semigroup.  Then $\ker \sour_\ast\colon R\mathcal G\to R\mathcal G^0$ is generated as a right ideal by the elements $1_V-1_{V\inv V}$ with $V\in X$ for any unital ring $R$.     
\end{Lemma}
\begin{proof}
     Recall that $R\mathcal G$ is generated by the characteristic functions of elements of the inverse semigroup $T$ of  compact bisections $V$ with $V\subseteq U$ for some $U\in S$ by Proposition~\ref{p:generate}.   If $V\in T$, then $\sour_*(1_V) = 1_{V\inv V}$.  Therefore, if $h=\sum_{i=1}^nr_i1_{V_i}\in \ker \sour_*$ with $V_i\in T$, then $\sum_{i=1}^nr_i1_{V_i\inv V_i}=0$, and so $h=\sum_{i=1}^nr_i(1_{V_i}-1_{V_i\inv V_i})$.
It follows that $\ker\sour_\ast$ is spanned over $R$ by the $1_V-1_{V\inv V}$ with $V\in T$.  

We prove $\ker \sour_\ast$ is the right ideal $\mathfrak R$ generated by elements of the form $1_V-1_{V\inv V}$ with $V\in X$.  Put $T'=\{V\in T\mid 1_V-1_{V\inv V}\in \mathfrak R\}\supseteq X$.  First note that $T'$ is closed under taking inverses since $1_{V\inv}-1_{VV\inv} = -(1_V-1_{V\inv V})1_{V\inv}\in \mathfrak R$ if $V\in T'$. Also, it is closed under products because \[1_{VW}-1_{W\inv V\inv VW} = (1_V-1_{V\inv V})1_W - (1_{W\inv}-1_{WW\inv} )1_{V\inv VW}\in \mathfrak R\] if $V,W\in T'$.   Finally, if $V\in T'$ and $U\subseteq V$ is a compact bisection, $VU\inv U=U$ and $V\inv VU\inv U =U\inv U$. Therefore, $1_U-1_{U\inv U}=(1_V-1_{V\inv V})1_{U\inv U}\in \mathfrak R$, and so $U\in T'$.  It follows that $T'=T$, and hence $\mathfrak R=\ker \sour_*$, as $T$ is the smallest inverse semigroup of compact bisections containing $X$ and closed downward under containment.
\end{proof}

Our next lemma is the technical heart of the proof and is adapted from~\cite[Proposition~3.6]{regulargpd}.  

\begin{Lemma}\label{l:isotropy.control}
Let $R$ be a unital ring and $\mathcal G$ an ample groupoid.  Suppose that $\mathcal G=\bigcup S$ where $S$ is a finitely generated inverse semigroup of compact open bisections.  If $R\mathcal G^0$ is a flat right $R\mathcal G$-module, then $\mathcal G$ is quasi-compact and $|\mathcal G^x_x|$ is invertible in $R$ for each $x\in \mathcal G^0$. \end{Lemma}
\begin{proof}
 Recall that $\mathcal G^0$ is compact and $R\mathcal G$ is spanned by the characteristic functions of elements of the inverse semigroup $T$ of  compact bisections $V$ with $V\subseteq U$ for some $U\in S$ by Proposition~\ref{p:generate}. We claim that $R\mathcal G^0$ is projective. It suffices to show that it is a finitely presented  $R\mathcal G$-module, as finitely presented flat modules are projective~\cite[Theorem~3.56]{Rotmanhom}.  
There is a surjective $R\mathcal G$-module homomorphism $\sour_*\colon R\mathcal G\to R\mathcal G^0$.  Since $R\mathcal G$ is a unital ring, as a right module it is cyclic generated by $1_{\mathcal G^0}$. Thus, it suffices to show that $\ker \sour_*$ is finitely generated.  But this follows from Lemma~\ref{l:generating.set} because $S$ is finitely generated.  We conclude that $R\mathcal G^0$ is projective.

It follows that $\sour_\ast$ splits, and hence there is an idempotent $e\in R\mathcal G$ with $\ker \sour_\ast=eR\mathcal G$.  
Let $f=1_{\mathcal G^0}-e$.  Then $f$ is an idempotent such that $f\ker\sour_\ast=(1_{\mathcal G^0}-e)eR\mathcal G=0$. Therefore, $f1_U=f1_{U\inv U}$ for all $U\in T$, as $1_U-1_{U\inv U}\in \ker\sour_*$.  If \[f=\sum_{i=1}^rc_i1_{U_i}\] with $U_i\in T$ compact bisections, then $f$ can be nonzero on at most $r$ elements from the same $\ran$-fiber, a crucial point.

Let $a\in \mathcal G$ and $b\in \mathcal G$ with $\sour(a)=\ran(b)$.  We claim that $f(ab)=f(a)$. Indeed, let $U\in S$ with $b\in U$.  Then $(f1_U)(ab) = f(a)$ and $(f1_{U\inv U})(ab)=f(ab)$.  We conclude that $f(ab)=f(a)$ for all $b\in \ran^{-1}(\sour(a))$.

%Notice that if $U\in T$ and $V\subseteq \mathcal G^0$ is compact open, then $1_V\cdot 1_U = 1_{U\inv VU}$ in the right module $\mathbb Z\mathcal G^0$.
%Let $\ran_*\colon R\mathcal G\to R\mathcal G^0$ be the left $R$-module homomorphism induced by the local homeomorphism $\ran$.  So \[\ran_*(f)(x)=\sum_{\ran(\gamma)=x}f(\gamma).\]  Note that if $U\in \Gamma_c(\mathcal G)$, then $\ran_*(1_U) =1_{UU\inv}$.   It is well known and easy to check that $\ran_*$ is a left $R\mathcal G$-module homomorphism where $R\mathcal G$ acts on $R\mathcal G^0$ via the rule \[f_1 \cdot g_1(x) = \sum_{\ran(\gamma)=x}f_1(\gamma)g_1(\sour(\gamma))\] for $f_1\in R\mathcal G$ and $g_1\in R\mathcal G^0$.  Notice that if $U\in \Gamma_c(\mathcal G)$ and $V\subseteq \mathcal G^0$ is compact open, then $1_U\cdot 1_V = 1_{UVU\inv}$.

We compute that $\sour_*(f) = \sour_*(1_{\mathcal G^0}-e) = \sour_*(1_{\mathcal G^0}) = 1_{\mathcal G^0}$, as $e\in \ker\sour_\ast$.  Let $x\in \mathcal G^0$.  Then
\begin{equation}\label{eq:augment}
1=\sour_*(f)(x) = \sum_{\sour(a)=x}f(a),
\end{equation}
 and so we can find $a\in \mathcal G$ with $\sour(a)=x$ and $f(a)\neq 0$.  Then, for every $b\in \ran^{-1}(x)$, we have $f(ab) =f(a)\neq 0$ by the previous paragraph.  Since $f$ is nonzero on at most $r$ elements with range $\ran(a)$, we deduce that $|\ran\inv(x)|\leq r$.   Thus $\mathcal G$ is uniformly bounded and hence has finite isotropy groups.
 The group $\mathcal G_x^x$ acts freely on the right of $\sour\inv (x)$, and so if we choose a transversal $\mathcal T$ for this action and use that $f(a'b)=f(a')$  for $\sour(a')=x$ and $b\in \mathcal G_x^x$, we obtain from \eqref{eq:augment}
\[1=\sum_{t\in \mathcal T}\sum_{g\in \mathcal G^x_x}f(tg)=|\mathcal G_x^x|\cdot \sum_{t\in \mathcal T} f(t).\] We conclude that $|\mathcal G_x^x|$ is a unit in $R$.  By Proposition~\ref{p:uniformly.bdd.vs.local.finite}, the finitely generated inverse semigroup $S$ is finite.  Therefore, $\mathcal G$ is a finite union of compact bisections, and hence quasi-compact.   This completes the proof.  
\end{proof}

\begin{proof}[Proof of Theorem~A]
To prove Theorem~A, first assume that $\mathcal G$ is AF.  Every compact principal ample groupoid has homological dimension $0$ by Theorem~\ref{t:hom.zero.prinipal}.  Therefore, $\mathcal G$ has homological dimension $0$ by Corollary~\ref{c:dir.hom.zero}.  

Conversely, suppose that $\mathcal G$ has homological dimension $0$.  Note that if $S$ denotes the inverse semigroup of compact bisections of $\mathcal G$, then $\mathcal G=\bigcup S$.  Now $S=\bigcup_{i\in D}S_i$ is the directed union of its finitely generated inverse subsemigroups $S_i$.  Put $\mathcal H_i=\bigcup S_i$, for $i\in D$.  Then $\mathcal H_i$ is an open subgroupoid and $\mathcal G=\bigcup_{i\in D}\mathcal H_i$ is a directed union.  By Corollary~\ref{c:shapiro.cor} each $\mathcal H_i$ has homological dimension $0$. Therefore, each $\mathcal H_i$ is a compact principal groupoid by Lemma~\ref{l:isotropy.control} since $\pm 1$ are the only units in $\mathbb Z$.
\end{proof}

\section{Proof of Theorem~B}
We begin with the proof of necessity in Theorem~B.  We first establish that the $R$-homological dimension of $\mathcal G$ is the minimum length of a flat resolution of the right $R\mathcal G$-module $R\mathcal G^0$.

\begin{Prop}\label{p:tensor.rg}
    Let $\mathcal G$ be an ample groupoid and $R$ a unital ring.  Let $M$ be a right $\mathbb Z\mathcal G$-module and $N$ a left $R\mathcal G$-module, both unitary.  Then, identifying $R\mathcal G$ and $R\otimes_{\mathbb Z}\mathbb Z\mathcal G$, there is an isomorphism $M\otimes_{\mathbb Z\mathcal G}N\cong (M\otimes_{\mathbb Z}R)\otimes_{R\mathcal G}N$, natural in $M$ and $N$.
\end{Prop}
\begin{proof}
Indeed, we have that $(M\otimes_{\mathbb Z}R)\otimes_{R\mathcal G}N\cong (M\otimes_{\mathbb Z\mathcal G}\mathbb Z\mathcal G\otimes_{\mathbb Z}R)\otimes_{R\mathcal G}N\cong M\otimes_{\mathbb ZG}R\mathcal G\otimes_{R\mathcal G} N\cong M\otimes_{\mathbb Z\mathcal G} N$ by Lemma~\ref{l:ideal.nice} and its left-right dual.
\end{proof}

An immediate corollary is  the following.
\begin{Cor}\label{c:stay.flat.es}
    Let $F$ be a flat unitary right $\mathbb Z\mathcal G$-module.  Then $F\otimes_{\mathbb Z} R$ is a flat unitary right $R\mathcal G$-module.
\end{Cor}
\begin{proof}
    By Proposition~\ref{p:tensor.rg}  $(F\otimes_{\mathbb Z} R)\otimes_{R\mathcal G}(-)\cong F\otimes_{\mathbb Z\mathcal G}(-)$, and hence is exact from which the result follows.
\end{proof}

The next corollary shows that if $M$ is a unitary left $R\mathcal G$-module, then $H_n(\mathcal G,M)\cong \Tor^{R\mathcal G}_n(R\mathcal G^0, M)$.  

\begin{Cor}\label{c:weakdim}
    Let $R$ be a unital ring, $\mathcal G$ an ample groupoid and $M$ a unitary left $R\mathcal G$-module.  Then $H_n(\mathcal G,M)\cong \Tor^{R\mathcal G}_n(R\mathcal G^0, M)$.  Consequently, the $R$-homological dimension of $\mathcal G$ is the shortest length of a flat resolution of the right $R\mathcal G$-module $R\mathcal G^0$.
\end{Cor}
\begin{proof}
    We claim that $B(\mathcal G)_\bullet\otimes_{\mathbb Z} R\to \mathbb Z\mathcal G^0\otimes_{\mathbb Z} R\cong R\mathcal G^0$ is a flat resolution over $R\mathcal G$.   Indeed,  $B(\mathcal G)_\bullet\otimes_{\mathbb Z} R$ is a chain complex of unitary flat $R\mathcal G$-modules by Corollary~\ref{c:stay.flat.es}.   We just need that the homology of this complex vanishes in degree $n\geq 1$.  Notice that $B_\bullet(\mathcal G)\to \mathbb Z\mathcal G^0$ is a resolution of the torsion-free (hence flat) abelian group $\mathbb Z\mathcal G^0$ by the torsion-free (hence flat) abelian groups $\mathbb Z\mathcal G^n$.  Therefore, $H_n(B_\bullet(\mathcal G)\otimes_{\mathbb Z} R)\cong \Tor^\mathbb Z_n(\mathbb Z\mathcal G^0,R)=0$, for $n\geq 1$.   We conclude that  $B(\mathcal G)_\bullet\otimes_{\mathbb Z} R\to R\mathcal G^0$ is a flat resolution.  Using this resolution to compute Tor, we obtain
    \begin{align*}
        \Tor_n^{R\mathcal G}(R\mathcal G^0,M)&\cong H_n((B_\bullet(\mathcal G)\otimes_{\mathbb Z} R)\otimes_{R\mathcal G} M)\\ &\cong  H_n(B_\bullet(\mathcal G)\otimes_{\mathbb Z\mathcal G} M)\cong H_n(\mathcal G,M)
    \end{align*}
    by Proposition~\ref{p:tensor.rg}.
The final statement follows from the isomorphism and~\cite[Proposition~8.17]{Rotmanhom}.
\end{proof}

Flatness of $R\mathcal G^0$ is then equivalent to $R$-homological dimension $0$.

\begin{Cor}\label{c:flat.rg}
An ample groupoid $\mathcal G$ has $R$-homological dimension $0$ if and only $R\mathcal G^0$ is a flat right $R\mathcal G$-module.
\end{Cor}

We can now prove necessity in Theorem~B.

\begin{Thm}\label{t:necessity}
    Let $R$ be a unital ring and $\mathcal G$ an ample groupoid of $R$-homological dimension $0$.  Then:
    \begin{enumerate}
        \item $\mathcal G$ is a directed union of quasi-compact open subgroupoids;
        \item The order of any finite subgroup of an isotropy group is a unit in $R$. 
    \end{enumerate}
\end{Thm}
\begin{proof}
    If $S$ denotes the inverse semigroup of compact bisections of $\mathcal G$, then $\mathcal G=\bigcup S$.  But, $S=\bigcup_{i\in D}S_i$ is the directed union of its finitely generated inverse subsemigroups $S_i$.  Set $\mathcal H_i=\bigcup S_i$ for $i\in D$.  Then $\mathcal H_i$ is an open subgroupoid and $\mathcal G=\bigcup_{i\in D}\mathcal H_i$ is a directed union.  By Corollary~\ref{c:shapiro.cor} each $\mathcal H_i$ has $R$-homological dimension $0$ and hence  is quasi-compact with isotropy groups having order a unit in $R$ by Corollary~\ref{c:flat.rg} and Lemma~\ref{l:isotropy.control}.  Since each finite subgroup of isotropy is contained in some $\mathcal H_i$, the result follows.
\end{proof}

To prove the converse to the above theorem, we proceed in steps following the ideas of~\cite{regulargpd}. %  A groupoid is a \emph{group bundle} if it consists of just isotropy.
A bisection $U\subseteq \mathcal G$ is \emph{invertible} if $UU\inv=\mathcal G^0=U\inv U$.

\begin{Thm}\label{t:grp.bundle}
    Let $\mathcal G$ be a quasi-compact ample group bundle such that the order of each isotropy group  is a unit in $R$.   Then $\mathcal G$ has $R$-homological dimension $0$.
\end{Thm}
\begin{proof}
    Note that $\mathcal G^0=\sour(\mathcal G)$ is compact and each isotropy group is finite.  Let $G$ be the group of all invertible compact bisections.   If $g\in \mathcal G$ and $U$ is a compact bisection containing $g$, then $V=U\cup (\mathcal G^0\setminus \sour(U))$ belongs to $G$ and contains $g$.  Thus $\mathcal G=\bigcup G$.  Since $\mathcal G$ is quasi-compact, we can cover it by finitely many elements of $G$.  These elements generate a finite subgroup $H$ of $G$, by Proposition~\ref{p:uniformly.bdd.vs.local.finite}, with $\mathcal G=\bigcup H$.

    Now $R\mathcal G$ and $R\mathcal G^0$ are cyclic $R\mathcal G$-modules generated by $1_{\mathcal G^0}$.  If we can find an idempotent $e\in R\mathcal G$ with $\sour_\ast(e)=1_{\mathcal G^0}$ and $e(1_V-1_{\mathcal G^0})=0$ for all $V\in H$, then it will follow from Lemma~\ref{l:generating.set} that $eR\mathcal G\cong R\mathcal G^0$, and therefore $R\mathcal G^0$ will be projective, hence flat. 

    We claim that $|H|$ is invertible in $R$.   If $p$ is a prime divisor of $|H|$, then we can find $V\in H$ of order $p$. Let $g\in V\setminus \mathcal G^0$. From $V^p=\mathcal G^0$, we deduce that  $g^p=\sour(g)$.   Thus $p\mid |\mathcal G_x^x|$, and hence $p$ is invertible in $R$.  Since all prime divisors of $|H|$ are units in $R$, we deduce that $|H|$ is a unit of $R$. 

    Let $e=\frac{1}{|H|}\sum_{V\in H}1_V$.  Then $e1_V=e$ for all $V\in H$, and so $e^2=e$ and $e(1_V-1_{\mathcal G^0})=0$ for all $V\in H$.  This completes the proof that $R\mathcal G^0$ is flat, and hence $\mathcal G$ has $R$-homological dimension $0$ by Corollary~\ref{c:flat.rg}. 
\end{proof}

Next we consider the case of a quasi-compact groupoid whose orbits all have the same size. Recall that a quasi-compact groupoid has finite orbits.

\begin{Prop}\label{p:local.homeo.equal.orbit}
    Let $\mathcal G$ be a quasi-compact ample groupoid such that all orbits of $\mathcal G$ have the same size.  Let $\mathcal R$ be the orbit equivalence relation.  Then $\mathcal G^0/\mathcal R$ is compact and totally disconnected, and the projection $q\colon \mathcal G^0\to \mathcal G^0/\mathcal R$ is a local homeomorphism.
\end{Prop}
\begin{proof}
    Notice that $\mathcal R$ is the image of $(\ran,\sour)\colon \mathcal G\to \mathcal G^0\times \mathcal G^0$, and hence is closed since $\mathcal G$ is quasi-compact and $\mathcal G^0\times \mathcal G^0$ is Hausdorff.  Thus $\mathcal G^0/\mathcal R$ is Hausdorff by~\cite[Proposition 10.4.8]{bourbakitop}, and hence compact.  

    The map $q$ is clearly open since $q\inv q(U) =\ran( \sour\inv(U))$ for $U\subseteq \mathcal G^0$.   To see that it is locally injective, assume that all orbits have size $n$.  Fix $x\in \mathcal G^0$.  Let $x=x_1,x_2\ldots, x_n$ be the orbit of $x$.    Choose compact bisections $U_1,\ldots, U_n$ such that $x\in \sour(U_i)$ and $x_i\in \ran(U_i)$.  We may assume that $U_1\subseteq \mathcal G^0$.  Since $\mathcal G^0$ is Hausdorff, we can find pairwise disjoint compact neighborhoods $V_1,\ldots, V_n$ of $x_1,\ldots, x_n$ in $\mathcal G^0$ and replace $U_i$ by $V_iU_i$ to ensure that the $\ran(U_i)$ are pairwise disjoint. Let $U=\bigcap_{i=1}^n\sour(U_i)$, an open subset of $\mathcal G^0$.  We claim that $q|_U$ is injective.  Indeed, if $y\in U$, then there exist $y_i\in \ran(U_i)$ in the orbit of $y$ for $i=1,\ldots, n$, and these are distinct since the $\ran(U_i)$ are pairwise disjoint. Moreover, $y_1=y$ as $U_1\subseteq \mathcal G^0$.  Hence $y=y_1,y_2,\ldots, y_n$ is the whole orbit of $y$.  Since $U\subseteq \sour(U_1)=U_1=\ran(U_1)$ is disjoint from $\ran(U_i)$ for $i>1$, we deduce that the only element from this orbit in $U$ is $y$.
 We conclude that $q$ is a local homeomorphism. Consequently, $\mathcal G^0/\mathcal R$ is totally disconnected. 
\end{proof}

\begin{Prop}\label{p:equal.orbit.case}
    Let $\mathcal G$ be a quasi-compact ample groupoid such that all orbits have the same size.  Suppose that each isotropy group has order a unit in $R$.  Then $R\mathcal G$ has $R$-homological dimension $0$.
\end{Prop}
\begin{proof}
    By Proposition~\ref{p:local.homeo.equal.orbit}, $q\colon \mathcal G^0\to \mathcal G^0/\mathcal R$ is a local homeomorphism and $\mathcal G^0/\mathcal R$ is compact and totally disconnected, where $\mathcal R$ is the orbit equivalence relation. Thus, we can find a section $t\colon \mathcal G^0/\mathcal R\to \mathcal G^0$ with compact open image $Y = t(\mathcal G^0/\mathcal R)$ as in the proof of Theorem~\ref{t:hom.zero.prinipal}.  Note that $Y$ is full and intersects each orbit exactly once, and so $\mathcal G^Y_Y$ is a  group bundle, clopen in $\mathcal G$.  Since $\mathcal G$ is quasi-compact, so is $\mathcal G^Y_Y$.  Therefore, $\mathcal G^Y_Y$ has $R$-homological dimension $0$ by Theorem~\ref{t:grp.bundle}.  If $M$ is a unitary left $R\mathcal G$-module, then $1_YM$ is a unitary left $R\mathcal G^Y_Y$-module.  Thus, $H_n(\mathcal G,M)\cong H_n(\mathcal G^Y_Y,1_YM)=0$ for $n\geq 1$ by Theorem~\ref{t:cheap.Morita}, and so $\mathcal G$ has $R$-homological dimension $0$.  
\end{proof}

The following lemma is~\cite[Lemma~3.9]{regulargpd}

\begin{Lemma}\label{l:lower.semi}
Let $\mathcal G$ be an ample groupoid and $k\geq 1$ be a natural number.  Then the set $O_k$ of elements $x\in \mathcal G^0$ whose orbit has at least $k$ elements is an open invariant subset.
\end{Lemma}

We can now handle general quasi-compact groupoids.

\begin{Prop}\label{p:suff}
Suppose that $R$ is a unital ring and $\mathcal G$ is a quasi-compact ample groupoid such that the order of any isotropy group is invertible in $R$.  Then $\mathcal G$ has $R$-homological dimension $0$.
\end{Prop}
\begin{proof}
%Since $\mathcal G$  is quasi-compact, its orbits are finite.
We can cover $\mathcal G$ by finitely many, say $r$, bisections.  Then, for any $x\in \mathcal G^0$, the size of the orbit of $x$ is at most $r$.  
Let $d_1<d_2<\cdots<d_m$ be the distinct orbit sizes.  We prove the result by induction on $m$.  If $m=1$, then all orbits of $\mathcal G$ have the same size and the result follows from Proposition~\ref{p:equal.orbit.case}.  So, assume that $m>1$ and the result holds for $m-1$.  Let $U=\{x\in \mathcal G^0\mid |\mathcal O_x|=d_m\}$.  Then $U$ is open and invariant by Lemma~\ref{l:lower.semi} once we note that $U=O_{d_m}$ in the notation of that lemma.   Let $C=\mathcal G^0\setminus U$.  %Then $I=R\mathcal G^U_U$ is an ideal of $R\mathcal G$ and $R\mathcal G/I\cong R\mathcal G^C_C$.  
Note that $\mathcal G^U_U=\ran\inv (U)$ is open and $\mathcal G^C_C=\ran\inv (C)$ is closed and hence quasi-compact.  By construction, the orbits sizes in $\mathcal G^C_C$ are $d_1<d_2<\cdots<d_{m-1}$, and so by induction $\mathcal G^C_C$ is of $R$-homological dimension $0$.     We next show that $\mathcal G^U_U$ is of $R$-homological dimension $0$.

Let $V\subseteq U$ be compact open.  Then $V$ is clopen in $\mathcal G^0$, and so $\sour\inv(V)\subseteq \mathcal G$ is clopen, whence quasi-compact and open.  Thus $W=\ran(\sour\inv(V))$ is a compact open invariant subset of $U$ containing $V$.  Hence the compact open invariant subsets of $U$ are cofinal among compact open subsets of $U$.  Moreover, if $W$ is a compact open invariant subspace of $U$, then  $W$ is clopen in $\mathcal G^0$ (as the latter is Hausdorff), and so $\mathcal G^W_W=\ran\inv(W)$ is a clopen subgroupoid of $\mathcal G$, hence quasi-compact, with all orbits of size $d_m$ (as $W\subseteq U$ is invariant).  Hence $\mathcal G^W_W$ is of $R$-homological dimension $0$ by Proposition~\ref{p:equal.orbit.case}.  Since $G^U_U$ is the directed union of the quasi-compact open subgroupoids $\mathcal G^W_W$ where $W$ runs over the compact open invariant subsets of $U$, we conclude that $\mathcal G^U_U$ is of $R$-homological dimension $0$ by Corollary~\ref{c:dir.hom.zero}. 

Finally, the long exact sequence of Theorem~\ref{t:exact.seq} implies that $\mathcal G$ has $R$-homological dimension $0$ because $\mathcal G^U_U$ and $\mathcal G_C^C$ do.  Indeed,  \[0=H_n(\mathcal G^U_U,\mathbb ZG^U_U\cdot M)\to H_n(\mathcal G,M)\to H_n(\mathcal G^C_C,M/\mathbb Z\mathcal G^U_U\cdot M)=0\] is exact for any unitary $R\mathcal G$-module $M$.
\end{proof}

\begin{proof}[Proof of Theorem~B]
 Theorem~\ref{t:necessity} yields necessity.  For sufficiency, we apply Proposition~\ref{p:suff} and Corollary~\ref{c:dir.hom.zero}, keeping in mind a quasi-compact groupoid has finite isotropy groups.
\end{proof}

\bibliographystyle{abbrv}
\bibliography{standard2}
\end{document}